\documentclass{amsart}
\usepackage[utf8]{inputenc}
\usepackage[T1]{fontenc}
\usepackage[english]{babel}

\usepackage{amssymb,amscd}
\usepackage{times}
\usepackage{color}
\usepackage{enumerate}
\newcounter{intro}

\newtheorem{theo}[intro]{Theorem}

\newtheorem{thm}{Theorem}[section]
\newtheorem{lem}[thm]{Lemma}
\newtheorem{prop}[thm]{Proposition}
\newtheorem{cor}[thm]{Corollary}

\theoremstyle{remark}
\newtheorem{rem}[thm]{Remark}

\newtheorem*{merci}{Acknowledgements}

\numberwithin{equation}{section}   

\newcounter{counteroman}

\newcommand{\cref}[1]{Corollary~\ref{#1}}

\newcommand{\lref}[1]{Lemma~\ref{#1}}
\newcommand{\pref}[1]{Proposition~\ref{#1}}
\newcommand{\rref}[1]{Remark~\ref{#1}}
\newcommand{\tref}[1]{Theorem~\ref{#1}}

\newcommand{\sref}[1]{Section~\ref{#1}}

\newcommand{\R}{\mathbb{R}}
\newcommand{\C}{\mathbb{C}}

\newcommand{\bB}{\mathbb{B}}

\newcommand{\N}{\mathbb{N}}

\newcommand{\Z}{\mathbb{Z}}
\newcommand{\bS}{\mathbb{S}}

\newcommand{\bP}{\mathbb{P}}

\newcommand{\mcC}{\mathcal{C}}

\newcommand{\bpm}{\begin{pmatrix}}
\newcommand{\epm}{\end{pmatrix}}

\newcommand{\ps}[2]{\left\langle#1|#2\right\rangle}
\newcommand{\nm}[1]{\left|#1\right|}
\newcommand{\nnm}[1]{\left\|#1\right\|}
\newcommand{\nnnm}[1]{\left|\left|\left|#1\right|\right|\right|}

\newcommand{\nnmr}[1]{\bigl\|#1\bigr\|}

\newcommand{\rstg}{\mathring{Ric_g}}

\let\ve=\varepsilon
\let\vp=\varphi

\DeclareMathOperator{\SO}{SO}

\DeclareMathOperator{\Id}{Id}

\DeclareMathOperator{\Rm}{Rm}
\DeclareMathOperator{\Ric}{Ric}
\DeclareMathOperator{\scal}{Scal}

\DeclareMathOperator{\vol}{vol}

\DeclareMathOperator{\Y}{Y}

\def\and{\,\, \mathrm{and}\,\,}

\begin{document}

\title[A sphere theorem for three dimensional manifolds with integral pinched curvature]
{A sphere theorem for three dimensional manifolds\\ with integral pinched curvature}

\author{Vincent Bour}
\email{Vincent.Bour@gmail.com} 
\author{Gilles Carron}
\email{Gilles.Carron@univ-nantes.fr}
\address{Laboratoire de Math\'ematiques Jean Leray (UMR 6629), Universit\'e de Nantes, 
2, rue de la Houssini\`ere, B.P.~92208, 44322 Nantes Cedex~3, France}

\date{\today}
\begin{abstract}
In \cite{BC}, we proved a number of optimal rigidity results for Riemannian manifolds of dimension greater than four whose curvature satisfy an integral pinching. In this article, we use the same integral Bochner technique to extend the results in dimension three. Then, by using the classification of closed three-manifolds with nonnegative scalar curvature and a few topological considerations, we deduce optimal sphere theorems for three-dimensional manifolds with integral pinched curvature.
\end{abstract}
\maketitle
\section{Introduction}

A celebrated result of R. Hamilton is the classification of closed three dimensional manifolds $(M^3,g)$ endowed with a Riemannian metric with non negative Ricci curvature (see \cite{H1} for metric with positive Ricci curvature and \cite{H2} for the case of nonnegative Ricci curvature). The result is that \textit{such Riemannian manifold $(M^3,g)$ is either:
\begin{itemize}
\item a flat manifold,
\item isometric to a quotient of the Riemannian product $\R\times \bS^2$ where $\bS^2$ is endowed with a round metric of constant Gaussian curvature,
\item or diffeomorphic to a space form: there is finite group $\Gamma\subset O(4)$ acting freely on $\bS^3$ such that 
$M^3=\bS^3/\Gamma$.
\end{itemize}}

This classification has been obtained with the Ricci flow and this result is certainly the first main milestone in the success of the Ricci flow.

In dimension four, a similar result has been obtained by C.~Margerin \cite{Margerin}: \textit{a closed \mbox{$4$-manifold }$M^4$ carrying a Riemannian metric with positive scalar curvature and whose
curvature tensor satisfies :
\begin{equation}\label{margerinpinching}
\|W_g\|^2+\frac12 \|\rstg\|^2\le \frac{1}{24}\scal_g^2
\end{equation} 
is either :
\begin{itemize}
\item isometric to $\bP^2(\C)$ endowed with the Fubini-Study metric $g_{FS}$,
\item isometric to a quotient of $\R\times \bS^3$ where  $\bS^3$ is endowed with the round metric of constant sectional curvature.
\item or diffeomorphic to $\bS^4$ or $\bP^4(\R)$.
\end{itemize}}
In the inequality (\ref{margerinpinching}), $W_g$ denotes the Weyl tensor of the metric $g$ and
\begin{equation*}
  \rstg=\Ric_g-\frac14 \scal_g g
\end{equation*}
is the traceless Ricci tensor. In fact, the above curvature pinching (\ref{margerinpinching}) implies that the Ricci curvature of $g$ is non-negative.

This result was a generalization of the classification of closed Riemannian $4$-manifold $(M^4,g)$ with positive curvature operator in \cite{H2}. 
This classification is now valid in all dimensions thanks to the work of C. B\"ohm, B. Wilking \cite{BW} and has been generalized by S.Brendle and R.Schoen to others pinching conditions \cite{BS1,BS2}.

These rigidity results, among many others in Riemannian geometry, involve what is called a ``pointwise curvature pinching'' hypothesis. The curvature is supposed to satisfy some constraint at every point of the Riemannian manifold, and strong restrictions on the topology of the manifold follow. 

Some of these results have been extended to manifolds that only satisfy the constraint in an average sense, i.e. that satisfy an integral curvature pinching.  For instance, Margerin's results was extended by A. Chang, M. Gursky and P. Yang in  \cite{changall1,changall2}, they show that\textit{ if $(M^4,g)$ is a closed Riemannian manifold with positive Yamabe invariant satisfying
\begin{equation}\label{margerinint}
\int_M\left(\|W_g\|^2+\frac12 \|\rstg\|^2\right)dv_g\le \frac{1}{24}\int_M\scal_g^2dv_g,
\end{equation} 
then $(M^4,g)$ is either :
\begin{itemize}
\item conformal  to $\bP^2(\C)$ endowed with the Fubini-Study metric $g_{FS}$,
\item conformal to a quotient of $\R\times \bS^3$ where  $\bS^3$ is endowed with the round metric of constant sectional curvature.
\item or diffeomorphic to $\bS^4$ or $\bP^4(\R)$.
\end{itemize}}

We recall that the \textit{Yamabe invariant} of a closed Riemannian manifold $(M^n,g)$ is the conformal invariant defined as :
$$Y(M^n,[g])=\inf_{\substack{\tilde g=e^{f}g\\ f\in \mcC^\infty(M)}} \vol(M,\tilde g)^{\frac2n -1}\int_M \scal_{\tilde g}dv_{\tilde g}.$$

In fact, all the hypotheses of the theorem are conformally invariant: by using the Gauss-Bonnet formula, the condition (\ref{margerinint}) is equivalent to
$$\int_M \|W_g\|^2dv_g\le 4\pi^2 \chi(M)$$ where $ \chi(M)$ is the Euler characteristic of $M$.

In dimension three, integral versions of the result of Hamilton have also been proved. For instance according to G.~Catino and Z.~Djadli \cite{CD} or Y.~Ge, C-S.~Lin and G.~Wang \cite{Ge}, \textit{a closed $3-$manifolds $(M^3,g)$ with positive scalar curvature such that 
\begin{equation}\label{3Dpinching}
\int_M \|\rstg\|^2dv_g\le \frac{1}{24}\int_M \scal_g^2dv_g
\end{equation} 
is diffeomorphic to a space form.}

However, this result is not optimal, and doesn't contain a caracterization of the equality case.

The strategy of the proof of these two integral pinching sphere theorems is to solve a fully nonlinear PDE in order to find a conformal metric 
$\bar g=e^{2f}g$ that satisfies Margerin's pointwise pinching in dimension $4$ or that has positive Ricci curvature in dimension $3$. Then the conclusion follows from the original pointwise version of the theorem.

In \cite{BC}, we used a Bochner method to extend a theorem of M.~Gursky in dimension four (\cite{Gu0}) to all dimensions greater than four: \textit{if $(M^{n},g)$ is a Riemannian manifold of dimension greater that four with positive Yamabe constant such that 
\begin{equation}
\left(\int_M\nnmr{\rstg}^{n/2}dv_g\right)^{4/n}\leq \frac{1}{n(n-1)} \Y(M^n,[g])^2,\label{eq:degre1compactnorm}
\end{equation}
then
\begin{itemize}
\item either its first Betti number $b_1(M^n)$ vanishes,
\item or equality is attained in \eqref{eq:degre1compactnorm}, $b_1(M^n)=1$, and there is an Einstain manifold $(N^{n-1},h)$ with positive scalar curvature such that $(M^{n},g)$ is isometric (or conformal in dimension four) to a quotient of the Riemannian product: $\R\times N^{n-1}$.
\end{itemize}}

The proof of the first part or the result is essentially the following: 
we first prove that the strict integral pinching implies that a certain Schr\"odinger operator $\square_g$ (see its definition in  section 
\ref{sec:bochner}) is positive, then we prove by a Bochner method that the positivity of the operator forces harmonic forms to vanish.  We rewrote it in this way in \sref{sec:bochner}.

However, we couldn't extent the result for three dimensional manifolds.  

In this article, we obtain similar results in dimension three, for instance with a pinching involving the operator norm of the Schouten tensor $\displaystyle A_g=\Ric_g-\frac 14\scal_g g$ :
\begin{equation*}
\nm{\nm{\nm{A_g}}}=\max_{v\in T_x M\setminus\{0\}} \frac{\nm{A_g(v,v)}}{\nm{v^2}},
\end{equation*}
and we deduce sphere theorems for three dimensional manifolds that satisfy integral pinchings.

For instance we obtain the following theorem : 
\begin{theo}
If $(M^3,g)$ is a closed Riemannian manifold whose Schouten tensor satisfies
\begin{equation*}
  \left(\int_M\nnnm{A_g}^{\frac32}dv_g\right)^{\frac{2}{3}}\le \frac{1}{4}Y(M,[g])
\end{equation*}
then 
\begin{itemize}
\item either $M$ carries a flat Riemannian metric.
\item or $M$ is diffeomorphic to a space form $M\simeq \bS^3/\Gamma $,
\item or $M$ is diffeomorphic to $\bS^1\times \bS^2$ or to $\bS^1\times \bP^2(\R)$ or to $\SO(3)\#\SO(3)$.
\end{itemize}
\end{theo}

Alternative pinching results are presented in section \ref{sec:pinching}.  
We first prove that when the pinching holds, the Scrödinger operator $ \square_g$ is non negative. 
If $ \square_g$ is positive, it is also positive on all finite covers of the manifold, 
so  we obtain that the first Betti number of all finite covers of the manifold vanishes. 
In dimension three, this is in fact sufficient to caracterize the quotients of the sphere, according to the classification of closed Riemannian manifolds with nonnegative scalar curvature. 

Our proof use three main ingredients, the first and major one is the classification of closed Riemannian manifold with positive scalar curvature initiated by R. Schoen and S-T. Yau \cite{SY2}, M. Gromov and H-B. Lawson \cite{GLihes}, and achieved by the fundamental work of G. Perelman (\cite{P1,P2,P3}). The second one is the Bochner's type argument that we used in \cite{BC} (\sref{sec:bochner}) and the last one is  a topological observation about the virtual Betti number of connected sum (\sref{sec:topology}). 

In a certain extent, our argument is similar to the new proof of the conformal sphere theorem of A. Chang, M. Gursky and P. Yang found recently by
B-L. Chen and X-P. Zhu \cite{chenZhu2}. They used a modified version of the Yamabe invariant, invented by M. Gursky (\cite{Gu}), in order to apply their classification of closed $4$-manifold carrying a Riemannian metric with positive isotropic curvature (\cite{H3} and \cite{chenZhu1}).

\begin{merci}
We would like to thank F. Laudenbach and S. Tapie for helpful discussions. Moreover,
the authors are partially supported by the grants ACG: ANR-10-BLAN 0105 and GTO: ANR-12-BS01-0004.   
\end{merci}

\section{Topological considerations}\label{sec:topology}

\begin{thm}\label{thm:b1}
Let $M^n$ be a closed manifold of dimension $n\ge 3$ admitting a connected sum decomposition:
$$M=X_1\#X_2$$ where $X_1$ and $X_2$ are not simply connected and where $\pi_1(M)$ is residually finite, then 
$M$ has a finite normal covering $\widehat M\rightarrow M$ with positive first Betti number:
$$b_1(\widehat{M})\ge 1.$$
\end{thm}
\begin{rem}\label{rem:residually} We recall that a group $\pi$ is  residually finite if for any finite subset $A\subset \pi\setminus\{e\}$ there is a normal subgroup
$\Gamma\triangleleft \pi$  with finite index such that $$A\cap \Gamma=\emptyset.$$

A result\footnote{We are grateful to F. Laudenbach for these references.} of K.-W.  Gruenberg shows that a free product  $\pi=\Gamma_1*\Gamma_2$ of residually finite group, (both $\Gamma_j$ is residually finite) is residually finite (see \cite{Gr} or \cite{Magnus}).\end{rem}
\proof Let $\Sigma\subset M$ be an embedded $(n-1)$-sphere such that 
$$M\setminus \Sigma=(X_1\setminus \bB^n)\cup (X_2\setminus \bB^n)\,.$$
Let $p\in \Sigma$, we have 
$$\pi_1(M,p)=\pi_1(X_1,p)\star \pi_1(X_2,p).$$
We choose $c_i \colon [0,1]\rightarrow X_i\setminus \bB^n$ a non trivial loop based at $p$ such that
$c_i(0)=c_i(1)=p$, we can assume that 
$$t\in (0,1)\Rightarrow c_i(t)\not\in \Sigma.$$
We consider $[\gamma]=[c_1\star c_2]\in \pi_1(M,p)$
we have 
$$\gamma(t)=\begin{cases}c_1(2t)& \mathrm{if}\,\, t\in [0,1/2]\\
c_2(2t-1)& \mathrm{if}\,\, t\in [1/2,1]\\
\end{cases}$$
Because $\pi_1(M,p)$ is assumed to be residually finite, we can find a normal subgroup 
$$\Gamma\subset \pi_1(M,p)$$ of finite index not containing $[c_1]$ and  $[c_2]$. We consider the quotient of the universal cover $\widetilde{M}\rightarrow M$ by $\Gamma$:
$$\pi\colon\widehat{M}=\widetilde{M}/\Gamma\rightarrow M.$$
We consider a lift $\widehat \gamma : \R \to \widetilde{M}$ of the continuous path $t\mapsto\gamma\left(t\,\mathrm{mod}\, 1\right)$. That is to say :
$$\pi\left(\widehat\gamma(t)\right)=\gamma\left(t\,\mathrm{mod}\, 1\right).$$
We consider $\widehat{\Sigma}$ the lift of $\Sigma$ such that 
$$\widehat{\gamma}(0)\in \widehat{\Sigma}.$$
Let $\tau>0$ be the first positive time such that 
$$\widehat{\gamma}(\tau)=\widehat{\gamma}(0).$$
By construction, we know that $\tau\in \frac 12\N$. Because we have assumed that $[c_1]$ and $[c_2]$ are not in $\Gamma$, we have $\tau\ne \frac 12$. Hence for all $t\in\left(0,\tau\right)$,
$$\widehat{\gamma}(\tau)\notin\widehat{\Sigma}.$$
If we define $\widehat{\gamma}_{\mathrm red}\colon \R/\tau\Z\rightarrow \widehat{M} $ by 
$$\widehat{\gamma}_{\mathrm red}(t)=\widehat{\gamma}(t\mod \tau\Z)\ ,$$
then the intersection number between $\widehat{\gamma}_{\mathrm red}$ and $\widehat{\Sigma}$ is $\pm 1$. Hence
 $$H^1(\widehat{M},\Z)\ne\{0\}.$$
\endproof

\begin{rem} Using a deep result of W. L\"uck, 
we can give another, more analytical, proof. We can equipped $M$ with a Riemannian metric $g$. Let  
$\widetilde{M}\rightarrow M$ be the universal cover of $M$. According to the main result of the paper \cite{Luck}, we only need to show that the universal cover of $M$ carries some non trivial $L^2$ harmonic $1$-forms. But a lift of $\Sigma$ to $ \widetilde{M}$ separated $\widetilde{M}$ into two unbounded connected components (because $X_1$ and $X_2$ are non simply connected), hence $\widetilde{M}$ has at least two ends. Moreover the injectivity radius of $(\widetilde{M},g)$ is positive. According to Brooks \cite{brooks}, if $\pi_1(M)$ is non amenable, then the Laplace operator acting on functions on $\widetilde{M}$ has a spectral gap, hence by \cite[Proposition 5.1]{carronpedon}, $\widetilde{M}$ carries a non constant harmonic function $h$ with $L^2$ gradient $dh\in L^2$.  Hence the result holds when $\pi_1(M)$ is not amenable. But\footnote{We are grateful to S. Tapie for explaining this to us. See \cite[Lemma 2.28]{Moon} for a proof.} a non trivial free product $\pi_1(X_1,p)\star \pi_1(X_2,p)$ is amenable only for $\Z_2\star \Z_2$; as 
$\Z_2\star \Z_2$ contains a normal subgroup of index $2$ isomorphic to $\Z$, we see that in this remaining case, $M$ will have a two fold cover with first Betti number equals to $1$.
\end{rem}

\section{A Bochner result}\label{sec:bochner}

In this section, we prove  a Bochner result, which was almost contained in \cite{BC}: if the operator $\square_g=\Delta_g+\frac{n-2}{n-1}\rho_1$ is positive, the first Betti number must vanish, and the equality case is characterized.

\subsection{Preliminaries}
We consider a closed connected Riemannian manifold $(M^n,g)$ of dimension $n\ge 3$.

We denote by $\rho_1$ the lowest eigenvalue of the Ricci tensor of $g$. The function $\rho_1\colon M\rightarrow \R$ satisfies 
$$\forall x\in M,\ \forall v\in T_xM,\ \Ric_g(v,v)\ge \rho_1(x) g(v,v)\,.$$

Then, we denote by $\square_g$ the Schr\"odinger operator
$$\square_g:=\Delta_g+\frac{n-2}{n-1}\rho_1.$$

It is nonnegative if for any smooth function $\varphi$,
$$\int_M \left(\nm{\nabla\varphi}^2+\frac{n-2}{n-1}\rho_1\varphi^2\right) dv_g\geq 0,$$
or equivalently if its lowest eigenvalue is nonnegative.

\begin{lem}\label{lem:cover}
Let $(M^n, g)$ be a closed Riemannian manifold and let $\pi:\widehat M\to M$ be a cover of $M$.

If the operator $\square_g$ is nonnegative, then the operator $\square_{\pi^*g}$ is nonnegative.
\end{lem}
\begin{proof}
  Let $u\in C^\infty( M)$ be an eigenfunction associated to the lowest eigenvalue of $\square_{g}$. We can suppose that $u$ is positive. 
Then the function $u\circ\pi$ is an positive eigenfunction for $\square_{\pi^*g}$. By a principle due to W.F. Moss and J. Piepenbrink 
and D. Fisher-Colbrie and R. Schoen, (\cite{MP,FCS} or \cite[lemma 3.10]{PRS}), 
we know that the bottom of the spectrum of  $\square_{\pi^*g}$ is non negative.

\end{proof}

\begin{lem}\label{lem:usquare}
Let $(M^n, g)$ be a closed Riemannian manifold. If $\xi$ is a non trivial harmonic one-form (i.e. $d\xi=0$ and $\delta\xi=0$), then the function $u=\nm{\xi}^{\frac{n-2}{n-1}}$ satisfies in the weak sense
\begin{equation*}
\square_g u +f_\xi u= 0
\end{equation*}
where
\begin{equation*}
f_\xi=\begin{cases}\frac{\Ric(\xi,\xi)}{\nm{\xi}^2}-\rho_1+\frac 1{\nm{\xi}^2}\left(\nm{\nabla\xi}^2-\frac n{n-1}\nm{d\nm{\xi}}^2 \right)&\text{where } \xi\neq 0\\
0&\text{where } \xi=0\end{cases}
\end{equation*}
is nonnegative on $M$.
\end{lem}
\begin{proof}
The harmonic one-form $\xi$ satisfies both the Bochner equation
\begin{equation*}
\ps{\nabla^*\nabla\xi}{\xi}+\Ric(\xi,\xi)=0,
\end{equation*}
and the refined Kato inequality  (\cite[lemma 2]{yau}, \cite{branson}, \cite{kato})
\begin{equation*}
\frac n{n-1}\nm{d\nm{\xi}}^2\leq \nm{\nabla\xi}^2.
\end{equation*}
According to basic calculations (see \cite{BC}, section 6), the function $u_\epsilon=(\nm{\xi}^2+\ve^2)^{\frac{n-2}{2(n-1)}}$ satisfies
\begin{equation*}
\Delta_g u_\ve+\frac{n-2}{n-1}\left(\rho_1+f_\xi\right)\nm{\xi}^2u_\ve^{-\frac n{n-2}}=0
\end{equation*}
And since $\nm{\xi}^2u_\ve^{-\frac n{n-2}}=u\,.\left(\frac{\nm{\xi}^2}{\nm{\xi}^2+\ve^2}\right)^{\frac{n}{n-1}}\leq u$, the function $u$ satisfies 
\begin{equation*}
\square_g u+f_\xi u= 0,  
\end{equation*}
in the weak sense.
\end{proof}
  
Finally, a twisted product $\bS^1\underline{\times}\ N$ is the quotient of $\R\times N$ by the cyclic group generated by a diffeomorphism
$$(t,x)\mapsto \left(t+l,f(x))\right)$$
where $f\colon N\rightarrow N$ is a diffeomorphism of $N$ and $l$ is a positive number. If $f$ is isotopic to the identity map then this twisted product is diffeomorphic to $\bS^1\times N$.

\subsection{The Bochner result}

We can reformulate a part of the Bochner result in \cite{BC} as follows:
\begin{prop}\label{prop:bochner} Let $(M^n,g)$ be a closed Riemannian manifold of dimension $n\ge 3$. If the operator 
\begin{equation*}
\square_g=\Delta_g+\frac{n-2}{n-1}\rho_1
\end{equation*}
is non negative then
\begin{itemize}
\item either the first Betti number of $M$ vanishes : $\mathrm{b}_1(M^n)=0$,
\item or $M^n$ is isometric to a twisted product $\bS^1\underline{\times}\, N^{n-1}$ endowed with a warped product metric
$$(dt)^2+\eta^2(t) h$$
where $(N^{n-1},h)$ is a closed Riemannian manifold with non negative Ricci curvature. 
\end{itemize}
\end{prop}
\begin{proof}
We assume that the operator $\square_g$ is nonnegative. If $b_1(M)>0$, according to the deRham isomorphism and the Hodge theorem, we can find a non trivial harmonic one-form $\xi\in \mcC^\infty(T^*M)$ with integral periods:
$$d\xi=0,\ \delta\xi=0\quad\text{and}\quad\forall \gamma\in \pi_1(M)\,,\, \int_\gamma \xi\in \Z.$$
According to \lref{lem:usquare}, the function $u:=|\xi|^{\frac{n-2}{n-1}}$ satisfies in the weak sense
\begin{equation*}
\square_g u+f_\xi u = 0.
\end{equation*}
Integrating by parts, it follows that
\begin{equation*}
\int_M \left(\nm{\nabla u}^2+\frac{n-2}{n-1}\rho_1 u^2\right) dv_g \leq - \int_M f_\xi u^2 dv_g \leq 0.
\end{equation*}
But as $\square_g$ is nonnegative, we must have
\begin{equation*}
\int_M \left(\nm{\nabla u}^2+\frac{n-2}{n-1}\rho_1 u^2\right) dv_g= 0,
\end{equation*}
and therefore $f_\xi=0$. It implies that at any point $x$ where $\xi(x)\neq 0$,  equality is attained in the refined Kato inequality, and $\xi$ is an eigenvector for the lowest eigenvalue of the Ricci tensor.

As equality is attained in the refined Kato inequality, we know (see \cite[Proposition 5.1]{BC}) that the normal cover $(\widehat M,\widehat g)$  associated to the kernel of the morphism
\begin{equation*}
\begin{array}{ccc}
\pi_1(M) & \to & \Z \\
\gamma  & \mapsto &  \int_\gamma \xi,
 \end{array}
\end{equation*}
is isometric to a warped product
\begin{equation*}
(\R\times N^{n-1} ,(dt)^2+\eta^2(t) h),
\end{equation*}
where $(N^{n-1},h)$ is a closed Riemannian manifold. 

Moreover, still according to  \cite[Proposition 5.1]{BC}, the pullback $\pi^*\xi$ of $\xi$ on $\widehat M$ is a multiple of the derivative of the function
\begin{equation*}
\Phi(t,x)=\int_0^t \frac{dr}{\eta(r)^{n-1}}.
\end{equation*}

Therefore, $\pi^*\xi$ is a multiple of $\frac{dt}{\eta(t)^{n-1}}$, and it implies that $dt$ is an eigenvector associated to the lowest eigenvalue of the Ricci tensor.

Then, we write the Ricci tensor of the warped product metric $(dt)^2+\eta^2(t) h$: 
\begin{equation*}
\Ric=\left(\begin{array}{cc} 
  -(n-1)\frac{\eta''}{\eta}&0\\
  0&\frac{\Ric_{h}}{\eta^2} -\frac{(n-2)(\eta')^2+\eta''\eta}{\eta^2}\Id
\end{array}\right).
\end{equation*}

Since dt is an eigenvalue associated to the lowest eigenvalue of the Ricci tensor, we must have:
$$\Ric_{h}\ge -(n-2)\left(\eta''\eta-(\eta')^2\right)h.$$

If the Ricci tensor of $h$ has a nonpositive  eigenvalue, then $\eta''\eta-(\eta')^2=\eta^2\left(\ln(\eta)\right)''$ must be nonnegative, and the function $\ln(\eta)$ is convex.

But the function $\eta$ is $\ell$-periodic, where $\ell\Z$ is the range of the morphism 
\begin{equation*}
\begin{array}{ccc}
\pi_1(M) & \to & \Z \\
\gamma  & \mapsto &  \int_\gamma \xi.
 \end{array}
\end{equation*}
Therefore, either $Ric_h$ is positive, or $Ric_h$ is nonnegative and $\eta$ is constant.
\end{proof}
According to Lemma~\ref{lem:cover}, we obtain the following corollary:
\begin{cor}\label{cor:cover} Let $(M^n,g)$ be a closed Riemannian manifold of dimension $n\ge 3$. If  the  operator 
$$\square_g:=\Delta_g+\frac{n-2}{n-1}\rho_1$$
is non negative then
\begin{itemize} 
\item either the first Betti number of any finite normal cover of $M$ vanishes.
\item or $M$ has a finite cover diffeomorphic to a twisted product $\bS^1\underline{\times}\ N^{n-1}$, where $N^{n-1}$ carries a metric of nonnegative Ricci curvature.
\end{itemize}
\end{cor}

\section{A sphere theorem  for three-dimensional manifolds}

In this section, we show how to obtain a sphere theorem for integral pinched manifolds based on the classification of three-dimensional manifolds with nonnegative scalar curvature.

The inequality $\scal_g\ge n\rho_1$ is always true. Therefore, if the operator $\square_g$ is nonnegative then the 
operator
$$ \Delta_g+\frac{n-2}{n(n-1)}\scal_g$$
is also nonnegative. In dimension $3$ and $4$, it implies that the Yamabe operator
$$ \Delta_g+\frac{n-2}{4(n-1)}\scal_g$$
is nonnegative. 

As a result, if $\square_g$ is nonnegative on a compact manifold of dimension $3$ or $4$, we can find a metric $\tilde g=u^{\frac 4{n-2}}g$ conformal to $g$ which has nonnegative scalar curvature.

But closed three dimensional manifolds carrying a metric with nonnegative scalar curvature have been classified. After the results of R. Schoen and S-T. Yau \cite{SY2} or M. Gromov and H-B. Lawson \cite{GLihes}, this classification is a consequence of the solution of the Poincar\'e conjecture by G. Perelman (\cite{P1,P2,P3}). The Ricci flow with surgeries of Perelman also provides a direct proof of this classification.

If $M^3$ is a closed oriented three dimensional manifold which carry a metric of nonnegative scalar curvature, then either $M^3$ carries a flat metric or it admits a connected sum decomposition
$$M^3=S_1\#S_2\#\dots\#S_r\#\ell\left(\bS^1\times \bS^2 \right)\,\,,$$
where each $S_j$ with $j\geq 2$ has a non trivial finite fundamental group $\Gamma_j$ and is diffeomorphic to a lens space $S_j\simeq \bS^3/\Gamma_j$.

According to this, we can prove the following theorem:
\begin{thm}\label{thm:sphere}
If $(M^3,g)$ is a closed manifold such that the operator
$$\square_g=\Delta_g+\frac{1}{2}\rho_1$$
is non negative, then
\begin{itemize}
\item either $M$ carries a flat Riemannian metric.
\item or $M$ is diffeomorphic to a space form $M\simeq \bS^3/\Gamma $,
\item or $M$ is diffeomorphic to $\bS^1\times \bS^2$ or to $\bS^1\times \bP^2(\R)$ or to $\SO(3)\#\SO(3)$.
\end{itemize}
\end{thm}
\proof
We consider the oriented cover $\bar M^3$  of $M^3$.
 
Since $\square_g$ is nonnegative,  either $\bar M^3$ carries a flat metric or it admits a connected sum decomposition
$$\bar M^3=S_1\#S_2\#\dots\#S_r\#\ell\left(\bS^1\times \bS^2 \right)\,\,,$$
where each $S_j$ with $j\geq 2$  has a non trivial finite fundamental group $\Gamma_j$ and is diffeomorphic to a lens space $S_j\simeq \bS^3/\Gamma_j$.

Using \tref{thm:b1} and \rref{rem:residually}, if the first cohomology group of all finite normal covers of $\bar M^3$ vanishes, then either $\bar M^3$ carries a flat metric,  or $(\ell,r)=(0,1)$ and $M^3$ is a space form, or $(\ell,r)=(1,0)$ and $M^3$ is a quotient of $\bS^2\times\bS^1$.

Then, according to \cref{cor:cover}, if there exists a finite normal cover of $\bar M^3$ with positive first Betti number, $\bar M^3$ has a finite cover diffeomorphic to $N^2\underline{\times}\ \bS^1$, where $N^2$ carries a metric of positive Ricci curvature. 

Therefore, either $N^2$ carries a flat metric and $M^3$ also carries a flat metric, or $N^2$ is a quotient of $\bS^2$. In the second case, since $\bar M^3$ is oriented, $N^2$ is diffeomeorphic to $\bS^2$, and the diffeomorphism $f:\bS^2\to\bS^2$ defining the twisted product  $N^2\underline{\times}\ \bS^1$ preserves orientation, hence is homotopic to the identity map. It follows that $M^3$ is diffeomorphic to a quotient of $\bS^2\times\bS^1$. 
 
Finally, the conclusion follows from the observation that the quotients of $\bS^2\times\bS^1$ are diffeomorphic to $\bS^1\times \bS^2$, $\bS^1\times \bP^2(\R)$ or $\SO(3)\#\SO(3)$.
\endproof

\section{Conditions for the operator $\square_g$ to be nonnegative}\label{sec:pinching}

Finally, we give several integral pinching conditions under which the operator $\square_g$ is nonnegative and deduce the sphere theorem implied by \tref{thm:sphere}. 

\subsection{With a pinching involving the Schouten tensor} 

In dimension $3$, the Riemann curvature tensor of a Riemannian metric $g$ has the following decomposition:
\begin{equation*}
\Rm_g=A_g\wedge g,
\end{equation*}
where $\wedge$ is the Kulkarni-Nomizu product and $A_g$ is the Schouten tensor, which satisfies 
$$A_g=\rstg+\frac 1{12}\scal_gg=\Ric_g-\frac 14\scal_g g.$$

We denote by $a_1\le a_2\le a_3$ the eigenvalues of the Schouten tensor. Then
$$\rho_1=a_1+\frac14\scal_g$$ 
and therefore
$$\square_g=\Delta_g+\frac18\scal_g+\frac12 a_1=\frac18 L_g+\frac12 a_1,$$
where $L_g=\Delta_g+\frac 18\scal_g$ is the Yamabe operator.

We recall that the Yamabe constant $Y(M,[g])$ is the best possible constant in the Sobolev inequality:
\begin{equation}\label{SobolevYamabe}\forall \vp\in\mcC^\infty(M),\quad Y(M,[g])\,\|\vp\|_{L^6}^2\le \int_M\left(8|d\vp|^2+\scal_g \vp^2\right)dv_g.
\end{equation}

Another interpretation of the Yamabe constant is given in more geometrical terms:
\begin{equation*}
  Y(M,[g])=\inf_{\substack{ \tilde g=e^{f}g\\ f\in C^\infty(M)}} \vol(M,\tilde g)^{-\frac13}\int_M \scal_{\tilde g}dv_{\tilde g}.
\end{equation*}

We define $(a_{1})_-=\max\{-a_1,0\}$ and the operator norm $\nnnm{A_g}$ of the Schouten tensor by:
$$\nnnm{A_g}(x)=\max_{v\in T_xM\setminus\{0\}} \frac{\left|g(A_gv,v)\right|}{g(v,v)}$$

The following inequality holds:
$$(a_{1})_-\leq \nnnm{A_g}$$
and equality is attained if and only if $a_1+a_3\le 0$.

\begin{prop}
If $(M^3,g)$ is a closed Riemannian manifold whose Schouten tensor satisfies
\begin{equation*}
\nnm{(a_1)_-}_{L^{\frac32}}\le \frac{1}{4}Y(M,[g])\qquad\text{or}\qquad\nnnm{A_g}_{L^{\frac32}}\le \frac{1}{4}Y(M,[g])
\end{equation*}
then the operator $\square_g$ is nonnegative. 
\end{prop}
\begin{proof}
 Since $\rho_1=a_1+\frac14\scal_g$, we get from the Sobolev's type inequality (\ref{SobolevYamabe}) and the H\"older inequality 
  \begin{equation}\label{eq:schoutenIneq}\begin{split}
    8\int_M\left( \nm{d \vp}^2+\frac 12 \rho_1 \vp^2\right) dv_g &\ge\int_M \left(8|d\vp|^2+\scal_g \vp^2-4\left((a_{1})_-\right) \vp^2\right)dv_g\\
    &\ge \left( Y(M,[g])-4\nnm{(a_1)_-}_{L^{\frac32}}\right)\, \nnm{\vp}_{L^6}^2,
  \end{split}\end{equation} 
\end{proof}

According to \tref{thm:sphere}, we obtain

\begin{thm}
If $(M^3,g)$ is a closed Riemannian manifold whose Schouten tensor satisfies
\begin{equation*}
\nnm{(a_1)_-}_{L^{\frac32}}\le \frac{1}{4}Y(M,[g])\qquad\text{or}\qquad\nnnm{A_g}_{L^{\frac32}}\le \frac{1}{4}Y(M,[g])
\end{equation*}
then 
\begin{itemize}
\item either $M$ carries a flat Riemannian metric.
\item or $M$ is diffeomorphic to a space form $M\simeq \bS^3/\Gamma $,
\item or $M$ is diffeomorphic to $\bS^1\times \bS^2$ or to $\bS^1\times \bP^2(\R)$ or to $\SO(3)\#\SO(3)$.
\end{itemize}
\end{thm}

\begin{rem}\label{rem:3case}
  The third case corresponds to the equality case in \pref{prop:bochner}. Moreover, equality must also be attained in \eqref{eq:schoutenIneq}, therefore the function $u=\sqrt{\nm{\xi}}$ must satisfy the Yamabe equation $8 \Delta_g u+ \scal_g u =Y(M,[g]) u^5$.
  
It follows that in this case, the Riemannian metric $g$ lifts to a metric $$(dt)^2+\eta^2(t)h$$ on $\bS^1\times \bS^2$, where $h$ is the round metric of the sphere $\bS^2$ and  the function $\eta^{-1}$ is a Yamabe minimizer.
\end{rem}

\subsection{With a pinching involving the traceless Ricci tensor}
Let $r_1$ be the lowest eigenvalue of the traceless Ricci tensor defined by
$$\rstg:=\Ric_g-\frac13\scal_g g.$$
We always have
$$r_1^2\le \frac23 \|\rstg\|^2$$ where 
$ \|\rstg\|$ is the Hilbert-Schmidt norm of the traceless Ricci tensor, and equality occurs only when
the spectrum of $\rstg$ is $r_1$ and $-\frac{r_1}{2}$ with multiplicity two.

Then we have 
$$\square_g=\Delta_g+\frac16 \scal_g+\frac12 r_1.$$

We introduce the lowest eigenvalue $\mu(g)$ of the operator 
$$4\Delta_g+\scal_g\,.$$
This quantity has the remarkable property of being nonincreasing along the Ricci flow (see \cite{P1}). 
\begin{prop}
If $(M^3,g)$ is a closed Riemannian manifold whose traceless Ricci tensor satisfies
\begin{equation*}
    \|r_{1}\|_{L^{3}}\le \frac{1}{3}\sqrt{Y(M,[g])\, \mu(g)}\qquad\text{or}\qquad \| \rstg\, \|_{L^{3}}\le \frac{1}{\sqrt{6}}\sqrt{Y(M,[g])\, \mu(g)},
\end{equation*}
then the operator $\square_g$ is nonnegative. 
\end{prop}
\begin{proof}
For any $\vp\in C^\infty(M)$ we have by Hölder inequality
\begin{align*}
Y(M,[g])\, \mu(g)\, \|\vp\|^4_{L^3}& \le Y(M,[g])\,\|\vp\|^2_{L^6}\,
 \mu(g)\, \|\vp\|^2_{L^2}\\
 &\le \left( \int_M \left[8|d\vp|^2+\scal_g \vp^2\right]dv_g\right)\left(\int_M \left[4|d\vp|^2+\scal_g \vp^2\right]dv_g\right)
\end{align*}
By using the inequality $\sqrt{AB}\le \frac{1}{2}\left(A+B\right)$, we obtain the Sobolev inequality 
 $$\forall \vp\in C^\infty(M),\quad
 \sqrt{Y(M,[g])\, \mu(g)}\, \|\vp\|^2_{L^3}\le\int_M \left[6|d\vp|^2+\scal_g \vp^2\right]dv_g.$$
Finally, since $\rho_1=r_1+\frac13\scal_g$, we have 
  \begin{align*}
   6 \int_M \left(\nm{d \vp}^2+\frac 12 \rho_1 \vp^2\right) dv_g \ge\int_M \left(6|d\vp|^2+\scal_g \vp^2-3 r_{1} \vp^2\right)dv_g\\
    \ge \left(\sqrt{Y(M,[g])\, \mu(g)}-3\nnm{r_1}_{L^{3}}\right)\, \nnm{\vp}_{L^3}^2,
\end{align*} 
\end{proof}
According to \tref{thm:sphere}, we obtain
\begin{thm}\label{rst}Let  $(M^3,g)$ be a closed Riemannian manifold and let $\mu(g)$ be the lowest eigenvalue of the operator $4\Delta_g+\scal_g$. If the traceless Ricci tensor satisfies 
  \begin{equation*}
    \|r_{1}\|_{L^{3}}\le \frac{1}{3}\sqrt{Y(M,[g])\, \mu(g)}\qquad\text{or}\qquad \| \rstg\, \|_{L^{3}}\le \frac{1}{\sqrt{6}}\sqrt{Y(M,[g])\, \mu(g)},
  \end{equation*}
then 
\begin{itemize}
\item either $M$ carries a flat Riemannian metric.
\item or $M$ is diffeomorphic to a space form $M\simeq\bS^3/\Gamma $,
\item or $M$ is diffeormophic to $\bS^1\times \bS^2$ or to $\bS^1\times \bP^2(\R)$ or to $\SO(3)\#\SO(3)$.
\end{itemize}
\end{thm}

\subsection{When the scalar curvature is nonnegative}
We denote by $\kappa_g$ the minimum of the scalar curvature :
$$ \kappa_g=\min_{x\in M} \scal_g(x).$$
\begin{prop}
If $(M^3,g)$ is a closed Riemannian manifold whose traceless Ricci tensor satisfies
\begin{equation*}
    \|r_{1}\|_{L^{2}}\le \frac 13 Y^{\frac34}(M,[g])\, \kappa_g^{\frac14}\qquad \text{or} \qquad \| \rstg\, \|_{L^{2}}\le  \frac1{\sqrt{6}}Y^{\frac34}(M,[g])\, \kappa_g^{\frac14}
\end{equation*}
then the operator $\square_g$ is nonnegative. 
\end{prop}
\begin{proof}
  By using the same idea we have 
  \begin{align*}
    Y^{\frac34}(M,[g])\, \kappa_g^{\frac14}\, \|\vp\|^2_{L^4}& \le\left(Y(M,[g])\,\|\vp\|^2_{L^6}\right)^{\frac34}\,\left(
    \kappa_g \, \|\vp\|^2_{L^2}\right)^{\frac14}\\
    & \le\left(Y(M,[g])\,\|\vp\|^2_{L^6}\right)^{\frac34}\,\left(
    \int_M \scal_g \vp^2dv_g\right)^{\frac14}\\
    &\le \left( \int_M \left[8|d\vp|^2+\scal_g \vp^2\right]dv_g\right)^{\frac34}\left(\int_M\scal_g \vp^2dv_g\right)^{\frac14}
  \end{align*}
Then, using the inequality $A^{\frac34}B^{\frac14}\le \frac34 A+\frac14 B$, we obtain
  $$\forall \vp\in C^\infty(M),\quad
  Y^{\frac34}(M,[g])\, \kappa_g^{\frac14}\, \|\vp\|^2_{L^4}\le\int_M \left[6|d\vp|^2+\scal_g \vp^2\right]dv_g.$$
  Finally, we have
  \begin{align*}
   6 \int_M\left( \nm{d \vp}^2+\frac 12 \rho_1 \vp^2 dv_g\right) \ge\int_M \left(6|d\vp|^2+\scal_g \vp^2-3 r_{1} \vp^2\right)dv_g\\
    \ge \left( Y^{\frac34}(M,[g])\, \kappa_g^{\frac14}-3\nnm{r_1}_{L^{2}}\right)\, \nnm{\vp}_{L^3}^2,
\end{align*} 

\end{proof}
According to \tref{thm:sphere}, we obtain
\begin{thm}\label{scalar}Let  $(M^3,g)$ be a closed Riemannian manifold and let $\displaystyle\kappa=\min_{x\in M} \scal_g(x)$. If the traceless Ricci tensor satisfies 
  \begin{equation*}
    \|r_{1}\|_{L^{2}}\le \frac 13 Y^{\frac34}(M,[g])\, \kappa^{\frac14}\qquad \text{or} \qquad \| \rstg\, \|_{L^{2}}\le  \frac1{\sqrt{6}}Y^{\frac34}(M,[g])\, \kappa^{\frac14}
  \end{equation*}
then 
\begin{itemize}
\item either $M$ carries a flat Riemannian metric.
\item or $M$ is diffeomorphic to a space form $M\simeq \Gamma\backslash\bS^3 $,
\item or $M$ is diffeormophic to $\bS^1\times \bS^2$ or to $\bS^1\times \bP^2(\R)$ or to $\SO(3)\#\SO(3)$.
\end{itemize}
\end{thm}

\begin{rem}
\rref{rem:3case} is also valid for \tref{rst} and \tref{scalar}. Moreover, since equality must also be attained in the Hölder inequality, the function $u=\sqrt{|\xi|}$ must be constant. It implies that the Riemannian metric $g$ lifts to a product metric on $\bS^1\times \bS^2$, and that this metric is a Yamabe minimizer.

According to \cite{schoen}, it implies that the product metric on $\bS^1\times\bS^2$ is such that the length of the circle $\bS^1$ is below a certain bound.
\end{rem}

\end{document}